\documentclass[10pt,english]{article}
\usepackage{latexsym, amsfonts, amsmath, amssymb, amsthm, cite}

\usepackage[T1]{fontenc}
\usepackage[latin9]{inputenc}
\usepackage[letterpaper]{geometry}
\usepackage{amsmath,amsthm,setspace,amssymb,esint,fancyhdr,lastpage,extramarks,chngpage,enumerate}
\usepackage{url}
\usepackage{lipsum}

\usepackage{babel}
\newtheorem{theorem}{Theorem}
\newtheorem{lemma}[theorem]{Lemma}

\newtheorem{proposition}[theorem]{Proposition}

\theoremstyle{remark}

\newtheorem{remark}{Remark}

\newcommand{\sgm}{S_{g-1}(V=\square)}
\newcommand{\soodd}{S_{1,\text{o}}}
\newcommand{\stodd}{S_{2,\text{o}}}
\newcommand{\sumstar}{\sideset{}{^*}\sum}

\newcommand{\fpx}{\mathbb{F}_q((\frac{1}{x}))}

\numberwithin{theorem}{section} \numberwithin{equation}{section}
\begin{document}

\title{Improving the error term in the mean value of $L(\tfrac{1}{2},\chi )$ in the hyperelliptic ensemble}
\date{}
\author{Alexandra Florea}

\newcommand{\Addresses}{{
  \bigskip
  \footnotesize

 \textsc{Department of Mathematics, Stanford University, Stanford, CA 94305}\par\nopagebreak
  \textit{E-mail address}:  \texttt{amusat@stanford.edu}

}}

\maketitle

\begin{abstract}
Andrade and Keating computed the mean value of quadratic Dirichlet $L$--functions at the critical point, in the hyperelliptic ensemble over a fixed finite field $\mathbb{F}_q$. Summing $L(1/2,\chi_D)$ over monic, square-free polynomials $D$ of degree $2g+1$, the main term is of size $|D| \log_q |D|$ (where $|D|=q^{2g+1}$) and Andrade and Keating bound the error term by $|D|^{\frac 34+ \frac{\log_q(2)}{2}}$. For simplicity, we assume that $q$ is prime with $q \equiv 1 \pmod 4$. We prove that there is an extra term of size $|D|^{1/3} \log_q|D|$ in the asymptotic formula and bound the error term by $|D|^{1/4+\epsilon}$.

\end{abstract}
\section{Introduction}
In this paper, we study the first moment of quadratic Dirichlet $L$--functions at the critical point in the function field setting. Specifically, we are interested in 
\begin{equation}
 \sum_{D \in \mathcal{H}_{2g+1}} L \Big( \tfrac{1}{2}, \chi_D \Big), \label{study}
 \end{equation} when $g \to \infty$, where $\mathcal{H}_{2g+1}$ denotes the space of monic, square-free polynomials of degree $2g+1$ over $\mathbb{F}_q[x]$.  Andrade and Keating \cite{keatingandrade} found an asymptotic formula for the first moment \eqref{study}, when the cardinality of the ground field $q$ is fixed and 
 $q \equiv 1 \pmod 4$. They explicitly computed the main term, which is of the order $g q^{2g+1}$, and 
 obtained an error term of size $q^{g(3/2+\log_q 2)}$. In this paper, we consider the first moment when $q\equiv 1\pmod 4$ is prime for simplicity, and find that there is an extra term of size $gq^{2g/3}$ in the asymptotic formula. Then we bound the error term in \eqref{study} by $q^{g/2(1+\epsilon)}$ 
 for any $\epsilon >0$. 
\begin{theorem}
Let $q$ be a prime with $q \equiv 1 \pmod 4$. Then
$$ \sum_{D \in \mathcal{H}_{2g+1}} L \Big( \tfrac{1}{2}, \chi_D \Big) = \frac{C(1)}{2 \zeta(2)} q^{2g+1} \left[ (2g+1)+ 1+ \frac{4}{\log q} \frac{C'}{C} (1) \right] + q^{\frac{2g+1}{3}} R(2g+1)+ O(q^{g/2(1+\epsilon)}),$$ where $\mathcal{H}_{2g+1}$ denotes the space of monic square-free polynomials of degree $2g+1$ over $\mathbb{F}_q[x]$, 
 $$C(s)= \prod_P \left(1 - \frac{1}{(|P|+1)|P|^s} \right),$$  $\zeta$ is the zeta-function associated to $\mathbb{F}_q[x]$ and $R$ is a polynomial of degree $1$ that can be explicitly computed (see formula \eqref{extra}.) \label{impr}
\end{theorem}  
 Finding asymptotics for moments of families of $L$--functions over number fields is a well-studied problem. Considering the family of quadratic Dirichlet $L$--functions, Jutila \cite{jutila} computed the first moment in $1981$. He proved that
  $$\sum_{0<d \leq D} L \Big(\frac{1}{2}, \chi_d \Big) = \frac{P(1)}{4 \zeta(2)} D \left[ \log(D/\pi)+ \frac{\Gamma'}{\Gamma} (1/4) +4 \gamma -1 +4 \frac{P'}{P}(1) \right] + O( D ^{3/4+\epsilon}),$$ where $$P(s)= \prod_p \left( 1- \frac{1}{(p+1)p^s} \right). $$Goldfeld and Hoffstein \cite{goldfeld} improved the error bound to $D^{19/32+\epsilon}$. 
  Young \cite{young} considered the smoothed first moment and showed that the error term is bounded by $D^{1/2+\epsilon}$. 

The remainder term for the first moment of quadratic Dirichlet $L$-functions is conjectured to be of size $D^{1/4+\epsilon}$ in \cite{andersonrubinstein}. Our approach in bounding the remainder over function fields is similar to Young's method in \cite{young}, but in our setting, we are able to go beyond the square-root cancellation.

Jutila \cite{jutila} also computed the variance, and Soundararajan \cite{sound} computed the second and third moments, when averaging over real, primitive, even characters with conductor $8d$.  It is conjectured that 
$$   \sumstar_{0 < d \leq D} L \Big(\tfrac{1}{2}, \chi_d \Big)^k \sim C_k D (\log D)^{k(k+1)/2},$$  where the sum is over fundamental discriminants. Keating and Snaith \cite{keatingsnaith} conjectured a precise value for $C_k$, using analogies with random matrix theory.
There is another conjecture of Conrey, Farmer, Keating, Rubinstein and Snaith \cite{cfkrs} for the integral moments, and the formulas include all the principal lower order terms. The conjecture agrees with the computed first three moments.
 
In the function field setting, the analogous problem is to find asymptotics for 
\begin{equation}
 \frac{1}{ \mathcal{H}_{2g+1,q}} \sum_{D \in \mathcal{H}_{2g+1,q}} L \Big( \tfrac{1}{2}, \chi_D \Big)^k, \label{mom} \end{equation} as $|D| = q^{\deg(D)} \to \infty,$ where $\mathcal{H}_{2g+1,q}$ denotes the space of monic, square-free polynomials of degree $2g+1$ over $\mathbb{F}_q[x]$. Since we let $|D| \to \infty$, we can consider two limits: the limit $q \to \infty$ (and $g$ fixed), or $g \to \infty $ (and $q$ fixed). Katz and Sarnak \cite{katzsarnak}, \cite{katzsarnak2} used equidistribution results to relate the $q$--limit of \eqref{mom} to a random matrix theory integral, which was then computed by Keating and Snaith \cite{keatingsnaith}. 

Here, we are interested in the other limit, when $g \to \infty $ and $q$ is fixed. 
In analogy with the conjectured moments for the family of quadratic Dirichlet $L$--functions over number fields,  Andrade and Keating \cite{conjectures} conjectured asymptotic formulas for integral moments of $L(1/2,\chi_D)$, for $q$ fixed and $g \to \infty$. In the recent paper \cite{rubinstein}, Rubinstein and Wu provide numerical evidence for the conjecture in \cite{conjectures}. They numerically computed the moments for $k \leq 10, d \leq 18 $ (where $d=2g+1$) and various values of $q$ and compared them to the conjectured formulas. Their data suggest that the ratio of the actual moment to the conjectured moment goes to $1$ as $g $ grows. 

Note that we can also compute the shifted first moment $ \sum_{D \in \mathcal{H}_{2g+1}} L \Big( \tfrac{1}{2}+\alpha, \chi_D \Big),$ when $\alpha$ is in a small neighborhood of $0$. Working instead with the completed $L$--function $\Lambda(s, \chi_D) = q^{-g(1-2s)} L(s,\chi_D)$ (which satisfies the symmetric functional equations $\Lambda(s,\chi_D)= \Lambda(1-s,\chi_D)$ ), we can find asymptotics for
$$ \sum_{D \in \mathcal{H}_{2g+1}} \Lambda \Big( \frac{1}{2}+\alpha,\chi_D \Big), $$ and we compute two main terms, one of size $q^{2g+1-\alpha g}$ and another of size $q^{2g+1+\alpha g}$, and two secondary main terms, of size $q^{(2g+\alpha g)/3}$ and $q^{(2g-\alpha g)/3}$. The error will be bounded by $O(q^{g/2(1+\epsilon)})$. 

Using ideas developed in this paper, we are also able to compute the second and third moments of $L(1/2,\chi_D)$ in the hyperelliptic ensemble, and the answer agrees with the conjecture in \cite{conjectures}. Computing higher moments is work in progress. 

For the hyperelliptic ensemble, we note the results of Entin, Roditty-Gershon and Rudnick \cite{entin}, Faifman and Rudnick \cite{faifman} and Kurlberg and Rudnick \cite{kurlberg}. 
\newline

\textit{Acknowledgments.} I would like to thank Kannan Soundararajan for many helpful discussions and for the suggestions he has offered throughout this work. Also, I would like to thank Julio Andrade, Jon Keating and Zeev Rudnick for useful comments on this paper.
\section{Preliminaries and background} 
We first introduce the notation we will use throughout the paper and then we will provide some background information on $L$--functions over function fields, quadratic Dirichlet characters, and their connection to zeta functions of curves. 

Denote by $\mathcal{M}$ the set of monic polynomials in $\mathbb{F}_q[x]$, and by $\mathcal{M}_n$ and $\mathcal{M}_{\le n}$ 
the sets of monic polynomials of degree $n$ and degree at most $n$ respectively. Let $\mathcal{H}_d$ denote the space of monic, square-free polynomials over $ \mathbb{F}_q[x]$ of degree $d$. 
For a polynomial $f \in \mathbb{F}_q[x]$,  we will denote its degree by $d(f)$, and its norm $|f|$ is defined to be $q^{d(f)}$. 
The letter $P$ will always denote a monic, irreducible polynomial over $\mathbb{F}_q[x]$.

\subsection{Basic facts about quadratic Dirichlet characters over functions fields and their L-functions}
Most of the facts stated in this section are proven in \cite{rosen}.  For $\text{Re}(s) >1$, the zeta function of $\mathbb{F}_q[x]$ is defined by 
$$
 \zeta(s) = \sum_{f \text{ monic} } \frac{1}{|f|^s} = \prod_{\substack{P \text{ monic} \\ \text{ irreducible}}} \displaystyle \left( 1-\frac{1}{|P|^s} \right)^{-1}.
 $$ 
 Since there are $q^n$ monic polynomials of degree $n$, we see that
$$
 \zeta(s) = (1-q^{1-s})^{-1}.
 $$ 
 We also find it sometimes convenient to make the change of variables $u=q^{-s}$, and then write 
 $\mathcal{Z}(u)= \zeta(s)$, so that  $\mathcal{Z}(u)= (1-qu)^{-1}$. 

Assume that  $q$ is an odd prime with $q \equiv 1 \pmod 4$. For $P$ a monic irreducible polynomial, the quadratic residue symbol $ ( \frac{f}{P}  ) \in \{\pm 1\}$ is defined by 
$$ \Big( \frac{f}{P} \Big) \equiv f^{(|P|-1)/2} \pmod P,$$ for $(f,P)=1$.  If $P|f$, then $ \left( \frac{f}{P} \right) =0$. 
If $Q=P_1^{e_1} P_2^{e_2} \cdot \ldots \cdot P_r^{e_r},$ then the Jacobi symbol is defined by
$$ 
\Big( \frac{f}{Q} \Big) = \prod_{j=1}^r \Big( \frac{f}{P_j} \Big)^{e_j}.
$$ 
The following formula is the analogue of the quadratic reciprocity law over function fields.
If $A,B \in \mathbb{F}_q[x]$ are relatively prime, non-zero polynomials, then
$$ 
\Big( \frac{A}{B} \Big)  =\Big( \frac{B}{A} \Big) (-1)^{\frac{(q-1)}2 d(A) d(B)}.
$$
Since we are assuming that $q \equiv 1 \pmod 4$, the quadratic reciprocity law above gives $( \frac{A}{B} ) = ( \frac{B}{A})$. 
We will use this fact several times throughout the paper.

For $D$ a square-free, monic polynomial, define the quadratic character
$$ \chi_D(g)= \left( \frac{D}{g} \right).$$ For $f$ monic, non-squarefree, we similarly define the character $\chi_f$, which is given by the Jacobi symbol defined above.
Consider the $L$--function attached to the character $\chi_D$:
$$L(s, \chi_D)= \sum_{f \in \mathcal{M} } \frac{\chi_D(f)}{|f|^s}.$$ This converges for $\text{Re}(s)>1$. With the change of variables $u=q^{-s}$, we have
$$ L(s,\chi_D) = \mathcal{L}(u,\chi_D)= \sum_{f \in \mathcal{M} } \chi_D(f) u^{\text{deg}(f)} = \prod_{P} (1-\chi_D(P) u^{\text{deg}(P)})^{-1}.$$
Note that $\mathcal{L}(u,\chi_D)$ is a polynomial in $u$ of degree at most $\text{deg}(D)-1$. 

 One can show that for $d \geq 2$,
$$ |\mathcal{H}_d|= q^{d-1}(q-1).$$
If $D \in \mathcal{H}_{2g+1}$, then $\mathcal{L}(u,\chi_D)$ is a polynomial in $u$ of degree $2g$, and it satisfies the following functional equation:
\begin{equation}
\mathcal{L}(u,\chi_D)= (qu^2)^{g} \mathcal{L}(1/qu,\chi_D). \label{fe1}
\end{equation}
There is a connection between $L$--functions and zeta functions of curves. For $D \in \mathcal{H}_{2g+1}$, the affine equation $y^2=D(x)$ defines a projective and connected hyperelliptic curve $C_D$ of genus $g$ over $\mathbb{F}_q$. The zeta function associated to $C_D$ is defined by
$$ Z_{C_D} (u) = \exp \left( \sum_{r=1}^{\infty} N_r(C_D) \frac{u^r}{r} \right), $$ where $N_r(C_D)$ is the number of points on the curve $C_D$ over $\mathbb{F}_{q^r}$, including the point at infinity. Weil \cite{weil} showed that
$$ Z_{C_D} (u) = \frac{P_{C_D}(u)}{(1-u)(1-qu)},$$ where $P_{C_D}(u)$ is a polynomial of degree $2g$. Moreover, one can show that $P_{C_D}(u) = \mathcal{L}(u,\chi_D)$ (this was proven in Artin's thesis.) Weil \cite{weil} also proved the Riemann hypothesis for curves over function fields, so all the zeros of $\mathcal{L}(u,\chi_D)$ lie on the circle $|u|=q^{-1/2}$.
\subsection{Functional equation and preliminary lemma}

 Using \eqref{fe1}, one can show the following exact formula for $L(1/2,\chi_D)$. This is the analogue of the approximate functional equation in the number field setting.
\begin{lemma}
Let $D \in \mathcal{H}_{2g+1}$. Then
$$L \Big( \tfrac{1}{2}, \chi_D \Big) =  \sum_{f \in \mathcal{M}_{\leq g}} \frac{\chi_D(f)}{\sqrt{|f|}}+  \sum_{f \in \mathcal{M}_{\leq g-1}}  \frac{\chi_D(f)}{\sqrt{|f|}}.$$
\label{fe}
\end{lemma}
\begin{proof}
See Lemma $1$ in \cite{keatingandrade}.
\end{proof}
Using Lemma \ref{fe}, it follows that
\begin{equation}
\sum_{D \in \mathcal{H}_{2g+1}} L \Big( \tfrac{1}{2}, \chi_D \Big) =  \sum_{f \in \mathcal{M}_{\leq g}} \sum_{D \in \mathcal{H}_{2g+1}} \frac{\chi_D(f)}{\sqrt{|f|}}+  \sum_{f \in \mathcal{M}_{\leq g-1}}  \sum_{D \in \mathcal{H}_{2g+1}} \frac{\chi_D(f)}{\sqrt{|f|}}. \label{ecfunct}
\end{equation}
Now we prove the following lemma.
\begin{lemma}
For $f$ a monic polynomial in $\mathbb{F}_q[x]$, we have that
$$ \sum_{D \in \mathcal{H}_{2g+1}} \chi_D(f)= \sum_{C | f^{\infty}} \sum_{h \in \mathcal{M}_{2g+1-2 d(C)}} \chi_f(h) - q \sum_{C | f^{\infty}} \sum_{h \in \mathcal{M}_{2g-1-2 d(C)}} \chi_f(h),$$
where the first summation is over monic polynomials $C$ whose prime factors are among the prime factors of $f$. \label{firstpoint} \end{lemma}
\begin{proof}
Using the quadratic reciprocity law, since $q \equiv 1 \pmod 4$, we have that $\chi_D(f)= \chi_f(D)$.
Let $$A_f(u)= \sum_{\substack{D \text{ square-free}  \\ \text{monic}}} u^{d(D)} \chi_f(D) .$$ Using Euler products, we get that
$$ A_f(u)= \prod_{P \nmid f} (1+u^{d(P)} \chi_f(P) ) = \frac{ \mathcal{L}(u,\chi_f)} {\mathcal{L}(u^2, \chi_f^2) } = \frac{ \mathcal{L}(u,\chi_f)}{ \mathcal{Z}(u^2)\displaystyle  \prod_{P|f} (1- u^{2 d(P)} ) }$$  Since
$$ \prod_{P|f} (1-u^{2d(P)})^{-1}= \sum_{C | f^{\infty}} u^{2 d(C)},$$
and $\mathcal{Z}(u^2)=(1-qu^2)^{-1}$ and $\mathcal{L}(u,\chi_f)=\displaystyle \sum_{i=0}^{\infty} u^i \sum_{h \in \mathcal{M}_i} \chi_f(h)$, the conclusion follows.
\end{proof}
 \subsection{Outline of the proof}
 The main term will come from the contribution of square polynomials $f$ to \eqref{ecfunct}, just like in the number field case. We will express the sum over square polynomials $f$ as a contour integral. The integrand will have a pole at $u=1/q$, and by shifting contours, we could express the main term in terms of the residue at $u=1/q$, plus an error of size $q^{g(1+\epsilon)}$. By simply enlarging the contour of integration, $q^{g(1+\epsilon)}$ seems to be the best error we can hope for, so instead we leave the main term in its integral form and look at the contribution from non-square polynomials $f$ as well. In evaluating this term, we will use a form of Poisson summation over $\mathbb{F}_q[x]$. We will analyze the sum over square polynomials $V$ (where $V$ is the dual variable in the Poisson sum), and we'll show how this term combines with the main term, which will allow us to calculate their sum exactly, with no error term. When evaluating the sum over square polynomials $V$, we find an extra term of size $gq^{2g/3}$. Evaluating the sum over non-square polynomials $V$ will give an error of size $q^{g/2(1+\epsilon)}$.
  
In section \ref{poissonformula}, we will prove the Poisson summation formula over function fields, which relates different character sums. In section \ref{main}, we compute the main term. After using the Poisson summation formula, we consider the sum over square polynomials $V$, where $V$ is the dual variable in the Poisson formula. We will evaluate this sum in section \ref{sec:secondary}. In section \ref{sec:nonsquare}, we bound the contribution from non-square polynomials $V$. We'll show how the main term combines with the contribution from square polynomials $V$ and we conclude the proof of Theorem \ref{impr} in section \ref{last}.

We note that our approach is similar to Young's method of getting square-root cancellation for the smoothed first moment in \cite{young}. When evaluating individual terms, we can't get a better error term than $q^{g(1+\epsilon)}$; however, by matching terms, we can prove that the error term is bounded by $q^{g/2(1+\epsilon)}$. Note that in this setting, we can go beyond the square-root cancellation from the number field setting. This better result relies on repeatedly using the Riemann hypothesis over function fields and the fact that $\zeta$ has no zeros, hence $1/\zeta$, which appears frequently in our calculations, has no poles. This makes dealing with the square-free condition easier than over number fields and allows us to shift contours more than in the number field case, thus getting improved error terms.
 
 \section{Poisson summation formula}
 \label{poissonformula}
 In this section, we will prove the Poisson summation formula. For simplicity, we assume that the cardinality $q$ of $\mathbb{F}_q$ is a prime and $q \equiv 1 \pmod 4$. 
We begin by recalling the exponential function introduced by D. Hayes, \cite{hayes}.  Each $a \in \fpx$ can be written uniquely as
$$
a = \sum_{i=- \infty}^{\infty} a_i \left(\frac{1}{x} \right)^i ,
$$ 
with $a_i \in \mathbb{F}_q$, and such that all but finitely many of the $a_i$ with $i<0$ are nonzero. One can define the following valuation
$$ \nu(a) = \text{ smallest } i \text{ such that } a_i \neq 0.
$$ 
For $a \in \fpx$ define the exponential (see \cite{hayes})  
$$
e(a)= \displaystyle e^{2\pi i a_1/q},
$$ 
where $a_{1}$ is the coefficient of $1/x$ in the expansion of $a$. From \cite{hayes}, recall that for $a,b \in \fpx$, we have $e(a+b) = e(a) e(b)$. Also, for $A \in \mathbb{F}_q[x]$, $e(A)=1$. If $A,B,H \in \mathbb{F}_q[x]$ are such that $A \equiv B \pmod H$, then $e(A/H)=e(B/H)$.

Now define the generalized Gauss sum 
$$
G(u,\chi) = \displaystyle \sum_{V \pmod f} \chi(V) e \left(\frac{uV}{f}\right).
$$
The main result of this section is the following Poisson summation formula for Dirichlet characters.

\begin{proposition}
Let $f$ be a monic polynomial of degree $n$ in $\mathbb{F}_q[x]$ and let $m$ be a positive integer. If the degree $n$ of $f$ is even, then
\begin{equation}
\sum_{g \in \mathcal{M}_m} \chi_f(g) = \frac{q^m}{|f|} \left[ G(0,\chi_f) + (q-1) \sum_{V \in \mathcal{M}_{\leq n-m-2} } G(V,\chi_f) - \sum_{V \in \mathcal{M}_{n-m-1}} G(V,\chi_f) \right]. \label{even}
\end{equation}
If $n$ is odd, then
\begin{equation}
 \sum_{g \in \mathcal{M}_m} \chi_f(g) = \frac{q^m}{|f|} \sqrt{q} \sum_{V \in \mathcal{M}_{n-m-1}} G(V,\chi_f).  \label{odd} \end{equation}
\label{poissonmonic}
\end{proposition}
\begin{remark}
Note that when $f=P$ is an irreducible polynomial, using the fact that $G(V,\chi_P)=\left( \frac{V}{P} \right) |P|^{1/2}$ (which follows by using the next lemma), the formulas above agree with the formulas proven in Proposition $7$ in \cite{rudnickfrob}.
\end{remark}
Before proving the proposition above, we first state the following lemma, which allows us to compute $G(V,\chi_f)$.

\begin{lemma}
Suppose that $q \equiv 1 \pmod 4$. Then 
\begin{enumerate}
\item If $(f,g)=1$, then $G(V, \chi_{fg})= G(V, \chi_f) G(V,\chi_g)$.
\item Write $V= V_1 P^{\alpha}$ where $P \nmid V_1$.
Then 
 $$G(V , \chi_{P^i})= 
\begin{cases}
0 & \mbox{if }  i \leq \alpha \text{ and } i \text{ odd} \\
\phi(P^i) & \mbox{if }  i \leq \alpha \text{ and } i \text{ even} \\
-|P|^{i-1} & \mbox{if }  i= \alpha+1 \text{ and } i \text{ even} \\
\left( \frac{V_1}{P} \right) |P|^{i-1} |P|^{1/2} & \mbox{if } i = \alpha+1 \text{ and } i \text{ odd} \\
0 & \mbox{if } i \geq 2+ \alpha .
\end{cases}$$ 
\end{enumerate} \label{computeg}
\end{lemma}
\begin{proof}
This is analogous to Lemma $2.3$ in \cite{sound}. 
\end{proof}
\begin{proof}[Proof of Proposition \ref{poissonmonic}]
Note that if $m \geq n$, then $\sum_{g \in \mathcal{M}_m} \chi_f(g)=0$ if $f$ is not a square, and $\sum_{g \in \mathcal{M}_m} \chi_f(g) = q^m \frac{ \phi(f)}{|f|}$ if $f$ is a square. Combining this observation with the fact that $G(0,\chi_f)$ is nonzero if and only if $f$ is a square, in which case $G(0,\chi_f)= \phi(f)$ (which follows from Lemma \ref{computeg}), Proposition \ref{poissonmonic} follows.
 
 Now assume that $m<n$. We will first prove the following more general Poisson summation formula, which holds for any character $\chi \pmod f$. If $d(f)=n$ and $m<n$, then we will prove that
\begin{equation}
 \sum_{g \in \mathcal{M}_m} \chi(g) = \frac{q^m}{|f|} \sum_{d(V) \leq n-m-1} G(V,\chi) e \left( \frac{-V x^m}{f} \right), \label{initialpoisson}
  \end{equation} where the sum on the right hand side is over all polynomials $v$, not necessarily monic. Note that this form of Poisson summation also holds for character sums in intervals in $\mathbb{F}_q[x]$, as defined in \cite{rudnickkeating}, and the proof is similar to the proof of \eqref{initialpoisson}. For our purposes, considering the sum over monic polynomials is enough. 
 
  We begin the proof of \eqref{initialpoisson} by noticing that for any polynomial $g$ in $\mathbb{F}_q[x]$, 
  \begin{equation} \chi(g) = \frac{1}{|f|} \sum_{V \pmod f} e \left( \frac{Vg}{f} \right) G(-V,\chi) . \label{chig}
  \end{equation} Indeed, using the definition of the Gauss sum, 
\begin{equation}
\frac{1}{|f|} \sum_{V \pmod f} e \left( \frac{Vg}{f} \right) G(-V,\chi) = \frac{1}{|f|} \sum_{u \pmod f} \chi(u) \sum_{V \pmod f} e \left( \frac{V(g-u)}{f} \right) . \label{int1}
\end{equation} If $u \neq g$, then $\displaystyle \sum_{V \pmod f} e \left( \frac{V(g-u)}{f} \right) =0$, since we can pick a polynomial $h$ such that $e \left( \frac{h(g-u)}{f} \right) \neq 0$, and then
$$ e \left( \frac{h(g-u)}{f} \right)  \sum_{V \pmod f} e \left( \frac{V(g-u)}{f} \right)   = \sum_{V \pmod f} e \left( \frac{ (V+h)(g-u)}{f} \right) = \sum_{V \pmod f} e \left( \frac{V(g-u)}{f} \right) .$$
Hence the only nonzero term in \eqref{int1} is given by $u=g$, so \eqref{chig} follows.

For $g \in \mathcal{M}_m$, we write $g = x^m+u$, with $d(u) \leq m-1$. Using \eqref{chig} , we have
\begin{align}
\sum_{g \in \mathcal{M}_m} \chi(g) &= \frac{1}{|f|} \sum_{V \pmod f} G(-V,\chi) e \left( \frac{V x^m}{f} \right) \sum_{d(u) \leq m-1} e \left( \frac{Vu}{f} \right) \nonumber \\
&=  \frac{1}{|f|} \sum_{d(V) \leq n-m-1} G(-V,\chi) e \left( \frac{V x^m}{f} \right) \sum_{d(u) \leq m-1} e \left( \frac{Vu}{f} \right)  \nonumber \\
&+  \frac{1}{|f|} \sum_{n-m \leq d(V) \leq n-1} G(-V,\chi) e \left( \frac{V x^m}{f} \right) \sum_{d(u) \leq m-1} e \left( \frac{Vu}{f} \right) . \label{int2}
\end{align} Let $S_1$ be the first summand above, and $S_2$ the second. We first evaluate $S_1$. When $d(V) \leq n-m-1$ and $d(u) \leq m-1$, we have $ e \left( \frac{Vu}{f} \right)=1$, and since there are $q^m$ polynomials of degree less than or equal to $m-1$, then under $V \mapsto -V$, it follows that \begin{equation}
S_1= \frac{q^m}{|f|}  \sum_{d(V) \leq n-m-1} G(V,\chi) e \left( \frac{-V x^m}{f} \right). \label{s1} \end{equation} Now we'll show that $S_2=0$.  We write
$$ \sum_{d(u) \leq m-1} e \left( \frac{Vu}{f} \right) = 1+ \sum_{i=0}^{m-1} \sum_{c=1}^{q-1} \sum_{u \in c \mathcal{M}_i} e \left( \frac{Vu}{f} \right). $$
If $i \leq n-2-d(V)$, then $e \left( \frac{Vu}{f} \right) =1$. If $i \geq n- d(V)$, then $\displaystyle \sum_{u \in c \mathcal{M}_i} e \left( \frac{Vu}{f} \right) = 0$. This follows from the more general fact that if $a \in \mathbb{F}_q \left( \left( \frac{1}{x} \right) \right)$ and $\nu(a)>0$, then 
$$ \sum_{u \in c \mathcal{M}_i} e(au) = \begin{cases} q^i e(c x^i a) & \mbox{ if } \nu(a)>i \\
0 & \mbox{ otherwise. } \end{cases} $$ For a proof of this, see Lemma $3.7$ in \cite{hayes}. Since $n-1-d(V) \leq m-1$, combining all of the above it follows that
\begin{align*}
\sum_{d(u) \leq m-1} e \left( \frac{Vu}{f} \right) &= 1+ \sum_{i=0}^{n-2-d(V)} \sum_{c=1}^{q-1} q^i + \sum_{c=1}^{q-1} \sum_{u \in c  \mathcal{M}_{n-1-d(V)}} e \left( \frac{Vu}{f} \right) \\
&= 1+ (q-1) \frac{q^{n-1-d(V)}-1}{q-1} - q^{n-1-d(V)} = 0,
\end{align*} so $S_2=0$. Combining this with \eqref{int2} and \eqref{s1} concludes the proof of \eqref{initialpoisson}. 

Now using \eqref{initialpoisson} for $\chi_f$ a Dirichlet character, we get that
\begin{equation}
 \sum_{g \in \mathcal{M}_m} \chi_f(g) = \frac{q^m}{|f|} \left[ G(0, \chi_f) + \sum_{d(V) \leq n-m-2} G(V,\chi_f) e \left( \frac{-V x^m}{f} \right) + \sum_{d(V) = n-m-1} G(V,\chi_f) e \left( \frac{-V x^m}{f} \right) \right]. \label{sum1} \end{equation} When $d(V) \leq n-m-2$, we have $e \left( \frac{-Vx^m}{f} \right)=1$, so $ \displaystyle \sum_{d(V) \leq n-m-2} G(V,\chi_f) e \left( \frac{-V x^m}{f} \right)  = \displaystyle \sum_{d(V) \leq n-m-2} G(V,\chi_f)$. We claim that 
\begin{equation} \sum_{d(V) \leq n-m-2 } G(V, \chi_f)   = \begin{cases} 0 & \mbox{ if $d(f)$ odd} \\
(q-1) \displaystyle \sum_{V \in \mathcal{M}_{\leq n-m-2}} G(V, \chi_f) & \mbox{ if $d(f)$ even}
 \end{cases} \label{small} \end{equation} and
 \begin{equation} \sum_{d(V)=n-m-1} G(V, \chi_f) e \left( \frac{-V x^m}{f} \right) = \begin{cases} \sqrt{q} \displaystyle \sum_{V \in \mathcal{M}_{n-m-1}} G(V,\chi_f) & \mbox{ if $d(f)$ odd} \\
 - \displaystyle \sum_{V \in \mathcal{M}_{n-m-1}} G(V, \chi_f) & \mbox{ if $d(f)$ even}
 \end{cases}  \label{big} \end{equation} We'll only prove \eqref{big} when $d(f) $ is odd, since the other case and \eqref{small} are similar. Note that $G(cV,\chi_f) = \chi_f(c^{-1}) G(V,\chi_f)$. When $V \in \mathcal{M}_{n-m-1}$ and $c \in \mathbb{F}_q^{*}$,  $e  \left( \frac{-cVx^m}{f} \right) = e(-c/q)$. Then
 \begin{align*}
\sum_{d(V)=n-m-1} G(V, \chi_f) e \left( \frac{-V x^m}{f} \right) &=   \sum_{V \in \mathcal{M}_{n-m-1}} \sum_{c=1}^{q-1} G(cV,\chi_f) e \left( \frac{-cV x^m}{f} \right)  \\
&= \sum_{V \in \mathcal{M}_{n-m-1}} G(V, \chi_f) \sum_{c=1}^{q-1} \chi_f(c^{-1}) e (-c/q) \\
&= \sqrt{q} \sum_{V \in \mathcal{M}_{n-m-1}} G(V, \chi_f).
\end{align*}
Now by Lemma \ref{computeg}, $G(0,\chi_f)$ is nonzero if and only if $f$ is a square (hence $d(f)$ even). Using this observation together with \eqref{sum1}, \eqref{small} and \eqref{big} yields the conclusion.
\end{proof}
\section{Setup of the problem} 
\label{setup}
Using Lemma \ref{firstpoint} and the functional equation \eqref{ecfunct}, we write

$$ \sum_{D \in \mathcal{H}_{2g+1}} L \Big( \tfrac{1}{2}, \chi_D \Big) = S_g +S_{g-1},  $$ where
$$S_g = \sum_{f \in \mathcal{M}_{\leq g}} \frac{1}{\sqrt{|f|}} \sum_{\substack{C | f^{\infty} \\ C \in \mathcal{M}_{\leq g}}} \sum_{h \in \mathcal{M}_{2g+1-2 d(C)}} \chi_f(h) - q  \sum_{f \in \mathcal{M}_{\leq g}} \frac{1}{\sqrt{|f|}} \sum_{\substack{C | f^{\infty} \\ C \in \mathcal{M}_{ \leq g-1}}} \sum_{h \in \mathcal{M}_{2g-1-2 d(C)}} \chi_f(h)  $$ and 
 $$S_{g-1}=  \sum_{f \in \mathcal{M}_{\leq g-1}} \frac{1}{\sqrt{|f|}} \sum_{\substack{C | f^{\infty} \\  C \in \mathcal{M}_{\leq g}}} \sum_{h \in \mathcal{M}_{2g+1-2 d(C)}} \chi_f(h) - q  \sum_{f \in \mathcal{M}_{\leq g-1}} \frac{1}{\sqrt{|f|}} \sum_{\substack{C | f^{\infty} \\  C \in \mathcal{M}_{ \leq g-1}}} \sum_{h \in \mathcal{M}_{2g-1-2 d(C)}} \chi_f(h) . $$ Note that in the equation above, when $C \in \mathcal{M}_g$, we express
 $$  \sum_{\substack{C | f^{\infty} \\  C \in \mathcal{M}_{g}}} 1 = \frac{1}{2 \pi i} \oint_{|u|=r_1} \frac{1}{u^{g+1}\prod_{P|f} (1-u^{d(P)})} \, du,$$ where $r_1<1$, so choosing $r_1=q^{-\epsilon}$, it follows that $ \sum_{\substack{C | f^{\infty} \\  C \in \mathcal{M}_{g}}} 1 \ll q^{\epsilon g}$. Then the term in the expression for $S_{g-1}$ corresponding to $C \in \mathcal{M}_g$ is bounded by $O(q^{g/2(1+\epsilon)})$. We rewrite
 \begin{equation}
 S_{g-1} = \sum_{f \in \mathcal{M}_{\leq g-1}} \frac{1}{\sqrt{|f|}} \sum_{\substack{C | f^{\infty} \\  C \in \mathcal{M}_{\leq g-1}}}  \left( \sum_{h \in \mathcal{M}_{2g+1-2 d(C)}} \chi_f(h) - q \sum_{h \in \mathcal{M}_{2g-1-2d(C)}} \chi_f(h) \right)+O(q^{g/2(1+\epsilon)}),  \label{sgminus1} \end{equation} and similarly
 $$ S_g= \sum_{f \in \mathcal{M}_{\leq g}} \frac{1}{\sqrt{|f|}} \sum_{\substack{C | f^{\infty} \\  C \in \mathcal{M}_{\leq g-1}}}  \left( \sum_{h \in \mathcal{M}_{2g+1-2 d(C)}} \chi_f(h) - q \sum_{h \in \mathcal{M}_{2g-1-2d(C)}} \chi_f(h) \right) +O(q^{g/2(1+\epsilon)}) .$$
 Now write $S_g= S_{g,\text{e}}+S_{g,\text{o}} +O(q^{g/2(1+\epsilon)})$ , where $S_{g,\text{e}}$ and $S_{g,\text{o}}$ denote the sum over monic polynomials $f$ of even and odd degree respectively. Similarly define $S_{g-1,\text{e}}$ and $S_{g-1,\text{o}}$. We focus on $S_{g-1}$. When $d(f)$ is odd, we use the Poisson summation formula as given in Proposition \ref{poissonmonic} for the sums over $h$ in \eqref{sgminus1}. 
 Then
 \begin{equation}
 S_{g-1,\text{o}} = q^{2g+1} \sqrt{q} \sum_{\substack{ f \in \mathcal{M}_{\leq g-1} \\ d(f) \text{ odd}}}  \frac{1}{|f|^{\frac{3}{2}}}  \sum_{\substack{C| f^{\infty} \\ C \in \mathcal{M}_{\leq g-1}}} |C|^{-2} \left( \sum_{V \in \mathcal{M}_{d(f)-2g-2+2d(C)}} G(V,\chi_f) - \frac{1}{q} \sum_{V \in \mathcal{M}_{d(f)-2g+2d(C)}} G(V,\chi_f)\right). \label{sgmodd} \end{equation}
 
For the term $S_{g-1,\text{e}}$, we use Proposition \ref{poissonmonic} again. Let $M_{g-1}$ be the term corresponding to the sum over $V=0$. Note that by Lemma \ref{computeg}, $G(0,\chi_f)$ is nonzero if and only if $f$ is a square, in which case $G(0,\chi_f)= \phi(f)$. Write $S_{g-1,\text{e}} = M_{g-1} + S_1$, where 
 \begin{equation*}
 M_{g-1}= q^{2g+1} \left(1-\frac{1}{q} \right) \sum_{\substack{f \in \mathcal{M}_{\leq g-1} \\ f= \square}} \frac{\phi(f)}{|f|^{\frac{3}{2}}} \sum_{\substack{C | f^{\infty} \\ C \in \mathcal{M}_{\leq g-1}}} \frac{1}{|C|^2}  ,\end{equation*} and
 \begin{align}
 S_1 &= q^{2g+1} \sum_{\substack{ f \in \mathcal{M}_{\leq g-1} \\ d(f) \text{ even}}}  \frac{1}{|f|^{\frac{3}{2}}}  \sum_{\substack{C| f^{\infty} \\ C \in \mathcal{M}_{\leq g-1}}} |C|^{-2} \bigg[ (q-1) \sum_{ V \in \mathcal{M}_{\leq d(f)-2g-3+2d(C)} } G(V, \chi_f) - \sum_{V \in \mathcal{M}_{d(f)-2g-2+2 d(C) }} G(V,\chi_f)  \nonumber \\
 & - \frac{q-1}{q}  \sum_{ V \in \mathcal{M}_{\leq d(f)-2g-1+2d(C)} } G(V, \chi_f) + \frac{1}{q} \sum_{ V \in \mathcal{M}_{d(f)-2g+2d(C)} } G(V, \chi_f)     \bigg].\label{s1}
 \end{align}
In the equation above, let $S_{g-1}(V=\square)$ be the sum over $V$ square and $S_{1}(V \neq \square)$ be the sum over $V$ non-square. Then $S_1 = S_{g-1}(V=\square)+ S_{1}(V \neq \square)$. 

When $V=l^2$, we write
\begin{align}
S_{g-1}(V=\square) &= q^{2g+1} \sum_{\substack{ f \in \mathcal{M}_{\leq g-1} \\ d(f) \text{ even}}}  \frac{1}{|f|^{\frac{3}{2}}}  \sum_{\substack{C| f^{\infty} \\ C \in \mathcal{M}_{\leq g-1}}} |C|^{-2} \bigg[ (q-1) \sum_{ l \in \mathcal{M}_{\leq \frac{d(f)}{2}-g-2+d(C)} } G(l^2, \chi_f) - \sum_{l \in \mathcal{M}_{\frac{d(f)}{2}-g-1+d(C)}} G(l^2,\chi_f)  \nonumber \\
& -\frac{q-1}{q}  \sum_{ l \in \mathcal{M}_{\leq \frac{d(f)}{2}-g-1+d(C)} } G(l^2, \chi_f)  + \frac{1}{q} \sum_{ l \in \mathcal{M}_{ \frac{d(f)}{2}-g+d(C)} } G(l^2, \chi_f) \bigg] .\label{sgsquare}
\end{align} 
Similarly define $S_{g}(V=\square)$. We'll evaluate these in section \ref{sec:secondary}.
Note that in equation \eqref{sgmodd}, when $d(f)$ is odd, $d(V)$ is also odd, so $V$ cannot be a square. Define $S_{g-1}(V \neq \square) = S_{g-1,\text{o}}+S_{1}(V \neq \square) $. Similarly define $S_g(V \neq \square)$. We'll bound $S_{g-1}(V \neq  \square)$ and $S_g(V \neq \square)$ in section \ref{sec:nonsquare}.
\begin{remark}
One way of explaining the term of size $q^{2g/3}$ in the asymptotic formula is by looking at those $C$ with $d(C)=g-d(f)/2$ in \eqref{sgsquare}. Then we have $l=1$, and $G(1,\chi_f)$ is nonzero if and only if $f$ is square-free, in which case $G(1,\chi_f)= \sqrt{|f|}$. Since we are summing over $C | f^{\infty}$ with $d(C) =g-d(f)/2$, when $C$ is square-free, we must have that $d(C) \leq d(f)$. Then  $d(f) \geq 2g/3$, which would contribute a term of size $q^{2g/3}$ in the formula \eqref{sgsquare}. We'll evaluate the term \eqref{sgsquare} in section \ref{sec:secondary} using analytic methods, so the ranges in which we sum will be less transparent. 
\end{remark}

   \section{Main term}
  \label{main}
  In this section, we evaluate the main terms $M_g$ and $M_{g-1}$. Recall that
  \begin{equation} 
  \label{5.1} 
  M_g= q^{2g+1} \left(1-\frac{1}{q} \right) \sum_{\substack{f \in \mathcal{M}_{\leq g} \\ f= \square}} \frac{\phi(f)}{|f|^{\frac {3}{2}}} \sum_{\substack{C | f^{\infty} \\ C \in \mathcal{M}_{\leq g-1}}} \frac{1}{|C|^2} ,
  \end{equation}
  and a similar expression holds for $M_{g-1}$.  
  
  Note that for any $f$ and any $\epsilon >0$, 
  $$ 
  \sum_{\substack{ {C|f^{\infty}} \\ {d(C)\ge g}}} \frac{1}{|C|^2} \le \frac{1}{q^{(2-\epsilon)g} }\sum_{ \substack{ { C|f^{\infty}} \\ {d(C)\ge g} } } \frac{1}{|C|^\epsilon} 
  \le \frac{1}{q^{(2-\epsilon)g}}  \prod_{P|f} \Big(1 -\frac{1}{|P|^{\epsilon}}\Big)^{-1}. 
  $$ 
  If now the degree of $f$ is at most $g$, then the product above is $\ll q^{g\epsilon}$ (since $q$ is 
  fixed, and $g$ is large).  Thus 
  $$ 
  \sum_{\substack{ {C|f^{\infty}} \\ {C \in \mathcal{M}_{\leq g-1}}}} \frac{1}{|C|^2} = 
  \sum_{C|f^{\infty}} \frac{1}{|C|^2} + O(q^{-(2-\epsilon)g} ) = 
  \prod_{P|f} \Big(1-\frac{1}{|P|^2}\Big)^{-1} + O(q^{-(2-\epsilon)g}).
  $$ 
  
  Write $f$ in \eqref{5.1} as $f= l^2$ and use the relation above.  Thus 
  \begin{align}
  M_g &= q^{2g+1}\Big(1-\frac 1q\Big) \sum_{l \in {\mathcal M}_{\le [ \frac{g}{2} ]}} \frac{\phi(l^2)}{|l|^3} \prod_{P|l} \Big(1-\frac{1}{|P|^2}\Big)^{-1} +O(q^{\epsilon g}) \nonumber\\ 
  &= 
  \frac{q^{2g+1}}{\zeta(2)} \sum_{l \in {\mathcal M}_{\le [ \frac{g}{2} ]}} \frac{1}{|l|} \prod_{P| l} \Big(\frac{|P|}{1+|P|}\Big) + O(q^{\epsilon g}). 
  \label{maintermg}
  \end{align}
  
  Below, we shall frequently make use of the following observation, which may be viewed as the function field analogue of Perron's formula.  If the power series $\sum_{n=0}^{\infty} a(n) z^{n}$ is absolutely convergent in $|z|\le r <1$ then 
 \begin{equation} 
 \label{perron} 
 \sum_{n=0}^{N} a(n)  = \frac{1}{2\pi i} \int_{|z|=r}  \Big(\sum_{n=0}^{\infty} a(n) z^n \Big) \frac{z^{-N-1}}{1-z} dz. 
 \end{equation} 
   Using \eqref{perron} in \eqref{maintermg} we obtain, for $r<1/q$,  
   $$
   M_g = \frac{q^{2g+1}}{\zeta(2)} \frac{1}{2\pi i} \int_{|u|=r}  \sum_{l \in {\mathcal M}} u^{d(l)} \prod_{P|l} \Big(\frac{|P}{1+|P|}\Big) \frac{(qu)^{-[ \frac{g}{2} ]}}{(1-qu)} \frac{du}{u} + O(q^{\epsilon g}). 
   $$ 
 Now, by multiplicativity, we may write 
 $$ 
 \sum_{l \in {\mathcal M}} u^{d(l)} \prod_{P|l} \Big(\frac{|P|}{1+|P} \Big) = \prod_P \Big(1 + \frac{|P|}{1+|P|} \frac{u^{d(P)}}{1-u^{d(P)}}\Big) =    
   \mathcal{Z}(u) \mathcal{C}(u)= \frac{{\mathcal C}(u)}{1-qu}, 
 $$
 where 
 \begin{equation} 
 \label{Pdef} 
 {\mathcal C}(u) = \prod_{P} \Big( 1- \frac{u^{d(P)}}{1+|P|} \Big).
 \end{equation}
 From its definition \eqref{Pdef} we see that ${\mathcal C}(u)$ is analytic in $|u|<1$, but we may further write 
 \begin{equation} 
 \label{Pdef2} 
 {\mathcal C}(u) = {\mathcal Z}(u/q)^{-1} \prod_P \Big ( 1+ \frac{u^{d(P)}}{(1+|P|)(|P|-u^{d(P)})}\Big), 
 \end{equation} 
 which furnishes an analytic continuation of ${\mathcal C}(u)$ to the region $|u|<q$.  
 
   From these remarks we conclude that for any $r<1/q$ 
   \begin{equation} 
   \label{Mgformula} 
   M_g = \frac{q^{2g+1}}{\zeta(2)} \frac{1}{2\pi i} \int_{|u|=r} {\mathcal C}(u) \frac{(qu)^{-[ \frac{g}{2} ]}}{(1-qu)^2} \frac{du}{u} + O(q^{\epsilon g}), 
   \end{equation} 
   and similarly 
   \begin{equation} 
   \label{Mgformula2} 
   M_{g-1} = \frac{q^{2g+1}}{\zeta(2)} \frac{1}{2\pi i} \int_{|u|=r} 
   {\mathcal C}(u) \frac{(qu)^{-[ \frac{g-1}{2} ]}}{(1-qu)^2} \frac{du}{u} + O(q^{\epsilon g}). 
   \end{equation} 
   The integrals in \eqref{Mgformula} and \eqref{Mgformula2} have double poles at $u=1/q$ and evaluating the residues we can obtain asymptotics for $M_g$ and $M_{g-1}$ with an error term of size $q^{(1+\epsilon)g}$.  
 We leave the expressions for $M_g$ and $M_{g-1}$ in the integral forms above, since we will show in the 
 next section how it matches up with other main terms leading finally to an asymptotic formula with an improved 
 error term.  
 
  \section{Contribution from $V$ square}
 \label{sec:secondary}
In this section, we will evaluate the terms $S_{g-1}(V=\square)$ and $S_g(V=\square)$. Recall from section \ref{setup} the formula \eqref{sgsquare} for $S_{g-1}(V=\square)$. The next lemma is the main result of this section.
\begin{lemma}
Using the same notation as before, we have that
 $$S_{g-1} (V=\square) = - \frac{q^{2g+1}}{\zeta(2)} \frac{1}{2 \pi i} \oint_{|u|=R} {\mathcal C}(u) \frac{(qu)^{-[ \frac{g}{2} ]}}{(1-qu)^2} \frac{du}{u} + q^{\frac{2g+1}{3}} P_1(g)+ O(q^{g/2(1+\epsilon)}),$$   and
 $$S_{g} (V=\square) = - \frac{q^{2g+1}}{\zeta(2)} \frac{1}{2 \pi i} \oint_{|u|=R}  {\mathcal C}(u) \frac{(qu)^{-[ \frac{g-1}{2} ]}}{(1-qu)^2}  \, \frac{du}{u} + q^{\frac{2g+1}{3}} P_2(g)+ O(q^{g/2(1+\epsilon)}),$$ with $1<R<q$, and $$\mathcal{C}(u) =\prod_P \left( 1 - \frac{u^{d(P)}}{|P|+1} \right).$$ Further, $P_1$ and $P_2$ are linear polynomials whose coefficients can be computed explicitly.  \label{tool}
 \end{lemma}
 
Before we prove Lemma \ref{tool}, we need the following results.
\begin{lemma} \label{bzw}
 For $|z|>1/q^2$, let $$\mathcal{B}(z,w) = \sum_{f \in \mathcal{M}} w^{d(f)} \prod_{P|f} \left(1 - \frac{1}{|P|^2 z^{d(P)}} \right)^{-1}  \left( \sum_{l \in \mathcal{M}} z^{d(l)} \frac{G(l^2,\chi_f)}{\sqrt{|f|}} \right) .$$
Then $$
\mathcal{B}(z,w) = \mathcal{Z}(z) \mathcal{Z}(w) \mathcal{Z}(qw^2z) \prod_P \mathcal{B}_P(z,w)$$ where
$$\mathcal{B}_P(z,w) = 1 + \frac{ w^{d(P)} - (w^2z)^{d(P)} |P|^2 - (wz^2)^{d(P)} |P|^2 + (w^3z^2)^{d(P)}|P|^2 + (w^2z)^{d(P)}|P| - (w^3z)^{d(P)}|P|}{z^{d(P)}|P|^2-1} .
$$
Moreover, $\prod_P \mathcal{B}_P(z,w)$ converges absolutely for $|w|<q|z|, |w|< 1/\sqrt{q}$, and $ |wz|<1/q$. 
\end{lemma}
\begin{proof}
 We rewrite $$B(z,w) = \sum_{l \in \mathcal{M}} z^{d(l)} \sum_{f \in \mathcal{M}} w^{d(f)} \frac{G(l^2,\chi_f)}{\sqrt{|f|}} \prod_{P|f} \left(1 - \frac{1}{|P|^2 z^{d(P)}} \right)^{-1} . $$ Recall Lemma \ref{computeg}, which evaluates $G(V,\chi_f)$ (a multiplicative function of $f$). Using this and an Euler product computation, it follows that
\begin{align}
  \sum_{f \in \mathcal{M}} w^{d(f)} \frac{G(l^2,\chi_f)}{\sqrt{|f|}} \prod_{P|f} \left(1 - \frac{1}{|P|^2 z^{d(P)}} \right)^{-1}  &= \mathcal{Z}(w) \prod_{P \nmid l } \left( 1+ \frac{w^{d(P)}}{(q^2z)^{d(P)}-1} - \frac{w^{2 d(P)}}{ 1- \frac{1}{(q^2z)^{d(P)}}} \right) \nonumber  \\& \prod_{P|l} (1-w^{d(P)}) \left(1 + \frac{1}{1- \frac{1}{(q^2z)^{d(P)}}} \sum_{i=1}^{\infty} \frac{w^{d(P)i} G(l^2, \chi_{P^i})}{|P|^{i/2}} .\right) \label{apl} \end{align}
Now we introduce the sum over $l$ and using Lemma \ref{computeg} again and manipulating Euler products, the expression for $\mathcal{B}(z,w)$ follows. The absolute convergence of $\prod_P \mathcal{B}_P(z,w)$ follows directly from the expression of $\mathcal{B}_P(z,w).$

\end{proof}
Using Lemma \ref{bzw} and an Euler product computation, we get the following.
\begin{lemma}
\label{poleqz}
Using the previous notation, we have that 
$$ \prod_P \mathcal{B}_P(z,w)=  \mathcal{Z} \left( \frac{w}{q^2z} \right) \mathcal{Z}(w^2)^{-1} \prod_P \mathcal{D}_P(z,w),$$ where
\begin{align*}
 \mathcal{D}_P(z,w) &= 1+ \frac{1}{(|P|^2z^{d(P)}-1)(1+w^{d(P)})} \bigg( -w^{2d(P)}-\frac{w^{3d(P)}}{|P|} + \frac{w^{d(P)}}{|P|^2z^{d(P)}}+|P|(w^2z)^{d(P)}+(w^2z)^{d(P)}-|P|^2 (wz^2)^{d(P)} \\
 &+(w^3z)^{d(P)}-|P|^2(w^2z^2)^{d(P)} \bigg). \end{align*}
Moreover, $\prod_P \mathcal{D}_P(z,w)$ converges absolutely for $|w|^2<q|z|, |w|<q^3|z|^2,|w|<1$ and $|wz|<q^{-1}$.

\end{lemma}

\begin{proof}[Proof of Lemma \ref{tool}]
We now turn to the proof of Lemma \ref{tool}. Recall the expression \eqref{sgsquare} for $S_{g-1}(V= \square)$. We use \eqref{perron} twice to express the sums over $l$ as contour integrals. Then
\begin{align}
S_{g-1}(V=\square)&= q^{2g+1} \frac{1}{2 \pi i} \oint_{|z|=r_1} \sum_{\substack{f \in \mathcal{M}_{\leq g-1} \\ d(f) \text{ even}}} \frac{1}{|f|} \sum_{\substack{C | f^{\infty} \\ C \in \mathcal{M}_{\leq g-1}}} \frac{1}{|C|^2}  \left(\sum_{l \in \mathcal{M}} z^{d(l)} \frac{G(l^2,\chi_f)}{\sqrt{|f|}} \right)  \frac{  (qz-1)}{(1-z) z^{\frac{d(f)}{2}+d(C) -g}} \left( 1- \frac{1}{qz} \right) \, dz, \label{int1}
\end{align} where we can pick $r_1= q^{-1-\epsilon}$.
We have
 \begin{equation} 
 \sum_{\substack{C | f^{\infty} \\ C \in \mathcal{M}_{\leq g-1}}} \frac{1}{|C|^2 z^{d(C)}} =  \sum_{C | f^{\infty}} \frac{1}{|C|^2 z^{d(C)}} -  \sum_{\substack{C | f^{\infty} \\ d(C) \geq g}} \frac{1}{|C|^2 z^{d(C)}} = \prod_{P|f} \left( 1- \frac{1}{|P|^2z^{d(P)}} \right)^{-1} - \sum_{\substack{C | f^{\infty} \\ d(C) \geq g}} \frac{1}{|C|^2 z^{d(C)}}.  \label{sumc} \end{equation}
Then we can write the integral \eqref{int1} as a difference of two integrals. We claim that the second integral, corresponding to the sum over $C$ with $d(C) \geq g$, is bounded by  $q^{g/2(1+\epsilon)}$. 
With the choice $r_1=1/q^{1+\epsilon}$, and using a similar argument as in section \ref{main} to bound the sum over $C$ with $d(C) \geq g$, we have
$$ \left|  \sum_{\substack{C | f^{\infty} \\ d(C) \geq g}} \frac{1}{|C|^2 z^{d(C)}} \right|  \ll q^{g \epsilon-g}.$$ Using the fact that $|G(l^2,\chi_f)|=G(l^2,\chi_f)$ and that for $l$ fixed, $G(l^2,\chi_f)$ is of size $\sqrt{|f|}$ on average (which follows from the proof of Lemma \ref{bzw}), we have
$$  \left|  \sum_{\substack{f \in \mathcal{M}_{\leq g-1} \\ d(f) \text{ even}}} \frac{1}{ z^{\frac{d(f)}{2}} |f| }  \frac{G(l^2,\chi_f)}{\sqrt{|f|}}  \sum_{\substack{C | f^{\infty} \\ d(C) \geq g}} \frac{1}{|C|^2 z^{d(C)}} \right| \ll q^{g/2(1+\epsilon)+g \epsilon-g}.$$ Since $ \sum_{l \in \mathcal{M}} z^{d(l)}$ converges when $|z|=r_1= q^{-1-\epsilon}$, it follows that the term corresponding to the sum over $d(C) \geq g$ is bounded by $q^{g/2(1+\epsilon)}$. Then
\begin{align*} 
S_{g-1}(V=\square) &= q^{2g+1} \frac{1}{2 \pi i} \oint_{|z|=\frac{1}{q^{1+\epsilon}}} \frac{(qz-1)z^g \left(1-\frac{1}{qz} \right)}{(1-z)} \sum_{\substack{f \in \mathcal{M}_{\leq g-1} \\ d(f) \text{ even}}} \frac{1}{|f|z^{\frac{d(f)}{2}}}  \prod_{P|f} \Big(1 - \frac{1}{|P|^2 z^{d(P)}} \Big)^{-1}  \Big( \sum_{l \in \mathcal{M}} z^{d(l)} \frac{G(l^2,\chi_f)}{\sqrt{|f|}} \Big) \, dz \\
&+O(q^{g/2(1+\epsilon)}).
\end{align*}
Using a variant of \eqref{perron} and Lemma \ref{bzw}, we have
\begin{equation}
S_{g-1}(V=\square)= q^{2g+1}  \left( \frac{1}{2 \pi i} \right)^2  \oint_{|w|=r_2} \oint_{|z|=\frac{1}{q^{1+\epsilon}}} \frac{(qz-1)z^g \left(1-\frac{1}{qz} \right) \mathcal{B}(z,w)}{w(1-z) (1-q^2w^2z) (q^2w^2z)^{[\frac{g-1}{2}]} }   \, dz \, dw + O(q^{g/2(1+\epsilon)}). \label{int3}
\end{equation}
Note that by Lemma \ref{bzw}, we must have $r_2<1/q$ and by \eqref{perron}, $|q^2w^2z|<1$ (since $|w|<1/q$ and $|z|<1$, this is already satisfied.) Again by Lemma \ref{bzw}, 
\begin{equation}
S_{g-1}(V=\square) =  -q^{2g+1} \left( \frac{1}{2 \pi i} \right)^2  \oint_{|w|=r_2} \oint_{|z|=\frac{1}{q^{1+\epsilon}}} \frac{z^g \left(1- \frac{1}{qz} \right) \prod_P \mathcal{B}_P(z,w)}{w (1-z) (1-qw)(1-q^2w^2z)^2 (q^2w^2z)^{[\frac{g-1}{2}]} } \, dz \, dw +O(q^{g/2(1+\epsilon)}).
\end{equation} 
Using Lemma \ref{poleqz}, we get
\begin{align*}
\sgm &=  -q^{2g+1}  \left( \frac{1}{2 \pi i} \right)^2  \oint_{|w|=r_2}  \oint_{|z|=\frac{1}{q^{1+\epsilon}}} \frac{ z^g  \left(1-\frac{1}{qz} \right) (1-qw^2) \prod_P \mathcal{D}_P(z,w) }{w (1-z) (1-qw) \left(1-\frac{w}{qz} \right)(1-q^2w^2z)^2 (q^2w^2z)^{[\frac{g-1}{2}]}} \, dz \, dw  \label{int3} \\
&+ O(q^{g/2(1+\epsilon)}).  \nonumber
\end{align*}
We shrink the contour $|z|=q^{-1-\epsilon}$ to $|z|=1/q^{3/2}$, and we don't encounter any poles. Then we enlarge the contour $|w|=r_2$ to $|w|=q^{-1/4-\epsilon}$, and we see that we encounter two simple poles, one at $w=1/q$ and one at $w=qz$. The pole at $w=1/q$ will give a term of size $q^g$, which will match the contribution from the main term, and the pole at $w=qz$ will give the term of size $g q^{(2g+1)/3}$. 
\begin{remark}
 Note that in equation \eqref{int3} above, we have flexibility in the way we shift contours. An alternative way is to first shift the contour $|w|=r_2$ in \eqref{int3} to $|w|=q^{-2\epsilon}$, encountering a simple pole at $w=1/q$ and a double pole at $w^2=1/(q^2z)$. We evaluate the pole at $w^2=1/(q^2z)$, and then shift the contour in the integral $|z|=q^{-1-\epsilon}$ to $|z|=q^{-3/2-\epsilon}$, encountering a pole at $z=q^{-4/3}$, which will give the term of size $gq^{(2g+1)/3}$.  
\end{remark}
Now we evaluate the residues at $w=1/q$ and $w=qz$, and write $\sgm = A+B+C+O(q^{g/2(1+\epsilon)})$, where
\begin{equation}
A= - q^{2g+1} \frac{1}{2 \pi i} \oint_{|z|=\frac{1}{q^{3/2}}} \frac{z^g \left( 1- \frac{1}{qz} \right) \prod_P \mathcal{B}_P(z,1/q) }{(1-z)^3 z^{[ \frac{g-1}{2} ]}} \, dz ,\label{mtg} \end{equation}
\begin{equation}
B= - q^{2g+1}  \frac{1}{2 \pi i} \oint_{|z|=\frac{1}{q^{3/2}}} \frac{z^g \left(1- \frac{1}{qz} \right) (1-q^3z^2)\prod_P \mathcal{D}_P(z,qz) }{ (1-z) (1-q^2z)(1-q^4z^3)^2(q^4z^3)^{[ \frac{g-1}{2} ]}} \, dz ,\label{secondpole}  \end{equation}
\begin{equation}
C=- q^{2g+1}  \left( \frac{1}{2 \pi i} \right)^2  \oint_{|w|=\frac{q^{-\epsilon}}{q^{1/4}}}  \oint_{|z|=\frac{1}{q^{3/2}}} \frac{z^g \left(1-\frac{1}{qz} \right) (1-qw^2) \prod_P \mathcal{D}_P(z,w)}{w   (1-z) \left( 1- \frac{w}{qz} \right)(1-qw)(1-q^2w^2z)^2 (q^2w^2z)^{[ \frac{g-1}{2} ]}} \, dz \, dw .\label{bound}
\end{equation}
Using Lemma \ref{poleqz}, note that \begin{equation}
 C \ll q^{g/2(1+\epsilon)}. \label{boundc}
 \end{equation}

Now we focus on evaluating the term $A$, given by \eqref{mtg} and then we will compute $B$ and show that it is of size $g q^{2g/3}$. 

Since $[g/2]+[(g-1)/2]=g-1$, we rewrite
\begin{equation*}
A= - q^{2g+1} \frac{1}{2 \pi i} \oint_{|z|=\frac{1}{q^{3/2}}} \frac{z^{1+[ \frac{g}{2} ]}  \left( 1- \frac{1}{qz} \right) \prod_P \mathcal{B}_P(z,1/q) }{(1-z)^3 } \, dz.
\end{equation*}
Now make the change of variables $z=1/(qu)$. Then the contour of integration will become the circle around the origin $|u|=\sqrt{q}$ and by Lemma \ref{bzw}, note that $\prod_P \mathcal{B}_P(1/(qu),q)$ has an analytic continuation for $q^{-2}<|u|<q$. We have
\begin{equation*}
A = - q^{2g+1} \frac{1}{2 \pi i} \oint_{|u|=\sqrt{q}} \frac{(1-u)  \prod_P \mathcal{B}_P \left( \frac{1}{qu},\frac{1}{q} \right)  \left( 1-\frac{1}{qu} \right)^{-1}  }{ (1-qu)^2 (qu)^{[ \frac{g}{2} ]}}  \, \frac{du}{u}.
\end{equation*}
Since $ 1-u = \mathcal{Z}(u/q)^{-1} = \prod_P \left(1 - \frac{u^{d(P)}}{|P|} \right)$ and $\left(1 - \frac{1}{qu} \right)^{-1} = \mathcal{Z} \left( \frac{1}{q^2u} \right) = \prod_P \left( 1- \frac{1}{|P|^2u^{d(P)}} \right)^{-1}$, by an Euler product computation, we get that
$$ (1-u)  \prod_P \mathcal{B}_P \left( \frac{1}{qu},\frac{1}{q} \right)  \left( 1-\frac{1}{qu} \right)^{-1}  = \frac{\mathcal{C}(u)}{\zeta(2)}, $$ with $\mathcal{C}(u)$ given in Lemma \ref{tool}. Then 
\begin{equation}
A = - \frac{q^{2g+1}}{\zeta(2)} \frac{1}{2 \pi i} \oint_{|u|=\sqrt{q}} \frac{\mathcal{C}(u)}{ (1-qu)^2 (qu)^{[ \frac{g}{2} ]}}\, \frac{du}{u}.
\label{matchingterm} 
\end{equation}
Now we focus on the term $B$ given by \eqref{secondpole}, which corresponds to the pole at $w=qz$. From the expression \eqref{secondpole}, we see that the integrand has a double pole at $z= q^{-4/3}$. By Lemma \ref{poleqz}, $\prod_P \mathcal{D}_P(z,qz)$ is absolutely convergent when $1/q^2<|z|<1/q$. Then
$$B = \text{Res}(z=q^{-4/3}) - q^{2g+1}  \frac{1}{2 \pi i} \oint_{|z|=\frac{1}{q^{1+\epsilon}}} \frac{z^g (1-q^3z^2)\prod_P \mathcal{D}_P(z,qz) \left(1- \frac{1}{qz} \right)}{ (1-z) (1-q^2z)(1-q^4z^3)^2(q^4z^3)^{[ \frac{g-1}{2} ]}} \, dz,$$ and note that the second term above is bounded by $O(q^{g/2(1+\epsilon)})$. Then
\begin{equation}
B = q^{2g/3} P_1(g) + O(q^{g/2(1+\epsilon)}), \label{secterm}
\end{equation} where $P_1$ is a linear polynomial whose coefficients can be computed explicitly. Combining equations \eqref{boundc}, \eqref{matchingterm} and \eqref{secterm}, we find the expression for $S_{g-1}(V=\square)$ in Lemma \ref{tool}. The formula for $S_g(V=\square)$ follows similarly. To be precise, we compute the term of size $gq^{2g/3}$ and rewrite it as
\begin{equation} q^{\frac{2g+1}{3}} P_1(g)+ q^{\frac{2g+1}{3}}P_2(g) = q^{\frac{2g+1}{3}} R(2g+1), \label{rpol} \end{equation} where $R$ is the linear polynomial 
\begin{equation}
R(x) = \frac{ \zeta(5/3) \zeta(7/3)}{9 q^{2/3} \zeta(4/3)^2 } \prod_P \mathcal{D}_P \left( \frac{1}{q^{4/3}}, \frac{1}{q^{1/3}} \right) \left[ \frac{x}{2} + \frac{\zeta(7/3)}{q^{4/3}} \left( -\frac{5}{2} -2q^{1/3}-2q+\frac{q^{4/3}}{2} \right)- \frac{2}{q^{4/3}} \frac{ \frac{1}{q^{4/3}} \frac{d}{dz} \prod_P \mathcal{D}_P(z,qz)}{ \prod_P \mathcal{D}_P(z,qz)} \rvert_{|z|=\frac{1}{q^{\frac{4}{3}}}} \right]. \label{extra} \end{equation} Moreover,
$$ \prod_P \mathcal{D}_P \left( \frac{1}{q^{4/3}},\frac{1}{q^{1/3}} \right) = \prod_P \left(1 - \frac{|P|^{4/3} + |P|^{2/3}+|P|^{1/3}+1}{(|P|^{4/3}+|P|)^2} \right), $$ and
$$ \frac{ \frac{1}{q^{4/3}} \frac{d}{dz} \prod_P \mathcal{D}_P(z,qz)}{ \prod_P \mathcal{D}_P(z,qz)} \rvert_{|z|=\frac{1}{q^{\frac{4}{3}}}} = - \sum_P \frac{d(P)(|P|-1)(|P|^{1/3}+1)}{(|P|^{1/3}-1)(|P|^{4/3}+|P|)^2}. $$
\end{proof}

     \section{Error from non-square $V$}
 \label{sec:nonsquare}
 In this section, we will bound the terms $S_g(V \neq \square)$ and $S_{g-1}(V \neq \square)$ by $q^{g/2(1+\epsilon)}$. Recall from section \ref{setup} that $S_{g-1}(V \neq \square) = S_{g-1,\text{o}} + S_1(V \neq \square)$, with $S_{g-1,\text{o}}$ given by \eqref{sgmodd} and $S_1(V \neq \square)$ the sum over non-square $V$ in equation \eqref{s1}. We'll prove the following lemma.
 \begin{lemma}
Using the same notation as before, we have 
 $$S_{g-1}(V \neq \square) \ll q^{g/2(1+\epsilon)} $$ and
 $$S_{g}(V \neq \square) \ll q^{g/2(1+\epsilon)} .$$
 \label{notsquare}
 \end{lemma}
  \begin{proof}
 We will focus on bounding $S_{g-1,\text{o}}$. The term $S_1(V \neq \square)$ is similarly bounded. In equation \eqref{sgmodd}, write $S_{g-1,\text{o}} = \soodd- \stodd$, where $\soodd$ corresponds to the first sum over $V$ with $d(V)=d(f)-2g-2-2d(C)$ and $\stodd$ corresponds to the sum over $V$ with $d(V)= d(f)-2g+2d(C)$. Let $n=d(f)$ and $i= d(C)$. Using the fact that
 $$ \sum_{\substack{C | f^{\infty} \\ C \in \mathcal{M}_i}} \frac{1}{|C|^2} = \frac{1}{2 \pi i} \oint_{|u|=r_1} \frac{1}{q^{2i}u^{i+1} \prod_{P|f} (1-u^{d(P)})} \, du,$$ with $r_1<1$, we rewrite
 \begin{equation}
  \soodd = q^{2g+1} \sqrt{q} \frac{1}{2 \pi i} \oint_{|u|=r_1} \sum_{\substack{n=0 \\ n \text{ odd}}}^{g-1} q^{-n}   \sum_{i=g+1-\frac{n+1}{2}}^{g-1} \frac{1}{q^{2i} u^{i+1}} \sum_{V \in \mathcal{M}_{n-2g-2+2i}} \sum_{f \in \mathcal{M}_n} \frac{G(V,\chi_f)}{ \sqrt{|f|}\displaystyle \prod_{P|f} (1-u^{d(P)})} \, du  \label{error2}. \end{equation}  As in the proof of Lemma \ref{bzw}, if we let $ \mathcal{B}(V;w,u) = \displaystyle \sum_{f \in \mathcal{M}} w^{d(f)}    \frac{G(V,\chi_f) }{ \sqrt{|f|} \displaystyle \prod_{P|f} (1- u^{d(P)})} $, we have 
  \begin{equation}
  \mathcal{B}(V;w,u) =  \mathcal{L}(w, \chi_V) \prod_{P \nmid V} \left(1 + \frac{ \left( \frac{V}{P} \right) (uw)^{d(P)}}{1-u^{d(P)}} - \frac{w^{2 d(P)}}{1-u^{d(P)}} \right) \prod_{P|V} \left( 1 + \frac{1}{1-u^{d(P)}} \sum_{i=1}^{\infty} \frac{G(V,\chi_{P^i}) w^{id(P)}}{|P|^{i/2}} \right). \label{bv} \end{equation} Let $\mathcal{B}_P(V;w,u)$ be the $P$-factor above. Note that $\prod_P \mathcal{B}_P(V;w,u)$ converges for $|wu|<1/q, |w|<1/\sqrt{q}$ and $|u|<1$. Now choose $|w|=q^{-1/2-\epsilon}$ and $|u|=r_1=q^{-\epsilon}$ and let $k$ be minimal such that  $|wu^k|<1/q$. Then we can write
\begin{equation}
 \prod_{P \nmid V}  \mathcal{B}_P(V;w,u) = \mathcal{L}(wu,\chi_V) \mathcal{L}(wu^2,\chi_V) \cdot \ldots \cdot \mathcal{L}(wu^{k-1}, \chi_V) \mathcal{C}(V;w,u), \label{bpv}
 \end{equation} where $\mathcal{C}(V;w,u)$ is given by an Euler product which converges absolutely when $|w|=q^{-1/2-\epsilon}$ and $|u|=q^{-\epsilon}$.
Now we have \begin{equation}
 \sum_{f \in \mathcal{M}_n}  \frac{G(V,\chi_f) }{ \sqrt{|f|} \displaystyle \prod_{P|f} (1- u^{d(P)})} = \frac{1}{2 \pi i} \oint_{\Gamma} \frac{\mathcal{B}(V;w,u)}{w^{n+1}} \, dw, \label{eq3}
 \end{equation} with $\Gamma$ a circle around the origin. In equation \eqref{bv}, when $|w|=q^{-1/2-\epsilon}$ and $|u|=q^{-\epsilon}$, note that $ \left| \prod_{P|v} \mathcal{B}_P(V;w,u) \right| \ll |V|^{\epsilon}.$ Combining equations \eqref{bv}, \eqref{bpv} and \eqref{eq3}, it follows that
\begin{equation}
 \left| \sum_{f \in \mathcal{M}_n}  \frac{G(V,\chi_f) }{\sqrt{|f|} \displaystyle \prod_{P|f} (1- u^{d(P)})}  \right| \ll q^{n/2(1+\epsilon)} \left| \mathcal{L}(w,\chi_V) \cdot \ldots \cdot \mathcal{L}(wu^{k-1},\chi_V) \right| |V|^{\epsilon}. \label{est} \end{equation}
Using Theorem $3.3$ in \cite{altug}, for a square-free polynomial $D$ of degree $2d$ or $2d+1$, we have
$$| \mathcal{L}(q^{-1/2},\chi_D) | \leq e^{ \frac{2d}{\log_q(d)} + 4 \sqrt{qd}}.$$ From the proof of Theorem $3.3$, it also follows that if $|w|<q^{-1/2}$, then $|\mathcal{L}(w,\chi_D) | \leq e^{ \frac{2d}{\log_q(d)} + 4 \sqrt{qd}}$. 
In equation \eqref{error2}, note that the degree of $V$ is odd. Then writing $V=AD^2$, with $D$ square-free, we relate $\mathcal{L}(w,\chi_V)$ to $\mathcal{L}(w,\chi_D)$ and get
$$|\mathcal{L}(w,\chi_V)| \ll e^{ \frac{n-2g+2i}{ 2 \log_q(n/2-g+i)} + 4\sqrt{q(n-2g+2i)}}.$$ Using the same bound for each $\mathcal{L}(wu^j,\chi_V)$ when $j \in \{0,1, \ldots, k-1\}$ and since $d(C)=i \leq g-1$, from equation \eqref{est} and the above it follows that
$$      \left| \sum_{f \in \mathcal{M}_n}  \frac{G(V,\chi_f) }{\sqrt{|f|} \displaystyle \prod_{P|f} (1- u^{d(P)})}  \right| \ll q^{n/2(1+\epsilon)} |V|^{\epsilon}.$$
Using this in \eqref{error2} and trivially bounding the sum over $V$, we get the desired bound for $S_{1,\text{o}}$. 
Similarly $\stodd \ll q^{g/2(1+\epsilon)}$ and then the conclusion follows.

 \end{proof}
 
 \section{Proof of Theorem \ref{impr}}
 Now we combine the results from the previous sections to prove Theorem \ref{impr}. We have
 \begin{align*}
 \sum_{D \in \mathcal{H}_{2g+1}} L \Big(\tfrac{1}{2},\chi_D \Big) &= M_g+S_g(V=\square)+S_g( V \neq \square) + M_{g-1} + S_{g-1}(V=\square)+S_{g-1}(V \neq \square).
 \end{align*}
 Using Lemma \ref{notsquare}, $S_{g-1}(V \neq \square) \ll q^{g(1/2+ \epsilon)}$ and $S_{g}(V \neq \square) \ll q^{g/2(1+ \epsilon)}$. Combining Lemma \ref{tool} and equation \eqref{Mgformula} gives
 \begin{align*}
 M_g+S_{g-1}(V=\square) &=  \frac{q^{2g+1}}{\zeta(2)} \frac{1}{2 \pi i} \oint_{|u|=r} \frac{\mathcal{C}(u)}{(1-qu)^2  (qu)^{[ \frac{g}{2} ]}} \, \frac{du}{u} -  \frac{q^{2g+1}}{\zeta(2)} \frac{1}{2 \pi i} \oint_{|u|=R} \frac{\mathcal{C}(u)}{(1-qu)^2  (qu)^{[ \frac{g}{2} ]}} \, \frac{du}{u}  \\ 
 &+ q^{\frac{2g+1}{3}} P_1(g)+ O(q^{g/2(1+\epsilon)}),
 \end{align*} where $r<1/q$ and $1<R<q$.  
 
  By \eqref{Pdef2} and the remark in section \ref{main}, $\mathcal{C}(u)$ has an analytic continuation for $|u|<q$, so between the circles $|u|=r$ and $|u|=R$, the integrand $\frac{\mathcal{C}(u)}{(1-qu)^2  (qu)^{[ \frac{g}{2} ]} u}$ only has a pole at $u=1/q$. Then 
 $$  \frac{q^{2g+1}}{\zeta(2)} \frac{1}{2 \pi i} \oint_{|u|=r} \frac{\mathcal{C}(u)}{(1-qu)^2  (qu)^{[ \frac{g}{2} ]}} \, \frac{du}{u} -  \frac{q^{2g+1}}{\zeta(2)} \frac{1}{2 \pi i} \oint_{|u|=R} \frac{\mathcal{C}(u)}{(1-qu)^2  (qu)^{[ \frac{g}{2} ]}} \, \frac{du}{u} = - \frac{q^{2g+1}}{\zeta(2)} \text{Res}(1/q).$$ We can easily compute the residue at $u=1/q$, so
 $$ M_g+S_{g-1}(V=\square) = \frac{q^{2g+1}}{\zeta(2)} \left[  \left( \left[ \frac{g}{2} \right] +1 \right) C(1) + \frac{C'(1)}{\log q} \right] + q^{\frac{2g+1}{3}} P_1(g)+O(q^{g/2(1+\epsilon)}),$$ where $C(s)= \mathcal{C}(u)$ with the change of variables $u=q^{-s}$.  In the same way 
 $$ M_{g-1} + S_g(V=\square) = \frac{q^{2g+1}}{\zeta(2)} \left[  \left( \left[ \frac{g-1}{2} \right] +1 \right) C(1)  + \frac{C'(1)}{\log q} \right]  + q^{\frac{2g+1}{3}} P_2(g)+O(q^{g/2(1+\epsilon)}).$$
Putting everything together and using equation \eqref{rpol}, Theorem \ref{impr} now follows. 
 \label{last}
\section{Removing the primality condition on $q$}
Note that in the proof of Theorem \ref{impr}, we assumed that $q$ is a prime. We can remove the condition on $q$ being a prime, but some changes need to be made in sections \ref{poissonformula} and \ref{sec:secondary}.

Let $q=p^k$ with $p$ prime and $q \equiv 1 \pmod 4$. Then following \cite{hayes}, for $a \in \mathbb{F}_q((1/x))$, we define the exponential by
$$ e(a) = e \Big( \frac{ \text{Tr}_{\mathbb{F}_{q} / \mathbb{F}_p} a_1}{p} \Big), $$ where $a_1$ is the coefficient of $1/x$ in the expansion of $a$. Then the Poisson summation formula \eqref{even} when $d(f)$ is even still holds. When $d(f)$ is odd, keeping the notation of Proposition \ref{poissonmonic}, we have
$$ \sum_{g \in \mathcal{M}_m} \chi_f(g) = \frac{q^m}{|f|}  \overline{\tau(q)} \sum_{V \in \mathcal{M}_{n-m-1}} G(V,\chi_f),$$ where $\tau(q)$ is the Gauss sum
$$ \tau(q) = \sum_{c=1}^{q-1} \chi_f(c) e \Big( \frac{ \text{Tr}_{\mathbb{F}_{q} / \mathbb{F}_p} c}{p} \Big).$$ By the Hasse-Davenport relation (see \cite{hasse}, chapter $11$), the Gauss sum above is $ \tau(q) = (-1)^{k-1} \tau(p)^k$, where $\tau(p) =  \sum_{c=1}^{p-1} \chi_f(c) e \Big( \frac{c}{p} \Big)$ is the usual Gauss sum and $  \tau(p) = \sqrt{p}$ if $p \equiv 1 \pmod 4$, $\tau(p) = i \sqrt{p}$ if $p \equiv 3 \pmod 4$. Note that $|\tau(q)|^2 =\sqrt{q}$. Since $q \equiv 1 \pmod 4$, we have that $\tau(q) = \pm \sqrt{q}$ ($\tau(q) = -\sqrt{q}$ exactly when $p \equiv 1 \pmod 4$ and $k$ is even, or when $p \equiv 3 \pmod 4 $ and $k \equiv 0 \pmod 4$).

When $q= p^k$, the other difference in section \ref{poissonformula} is in Lemma \ref{computeg}, for the value of $G(V,\chi_{P^{i}})$ when $i = \alpha+1$ is odd (where $P^{\alpha}$ is the highest power of $P$ dividing $V$.) In this case,
$$ G(V,\chi_{P^i}) = 
\begin{cases}
\left( \frac{V_1}{P} \right) |P|^{i-1} |P|^{1/2} & \mbox{if } d(P) \text{ even} \\
\left( \frac{V_1}{P} \right) |P|^{i-1} |P|^{1/2} \frac{1}{\sqrt{q}} \tau(q) & \mbox{if } d(P) \text{ odd.} 
\end{cases}$$ For all the other values of $i$, $G(V,\chi_{P^i})$ has the same expression as given in Lemma \ref{computeg}.

If $\tau(q)=\sqrt{q}$, then all the computations in section \ref{sec:secondary} remain the same. If $\tau(q)=-\sqrt{q}$, then some changes need to be made in section \ref{sec:secondary}. The function $\mathcal{B}(z,w)$ defined in Lemma \ref{bzw} is now $\mathcal{B}(z,w) = \mathcal{Z}(-w) \mathcal{Z}(z) \mathcal{Z}(qw^2z) \mathcal{C}(z,w)$, where $\mathcal{C}(z,w)$ is given by an Euler product that is analytic in a wider region. The computations will be similar to those in section \ref{sec:secondary}, and we do not carry them out here.

\bibliography{bibl}
\bibliographystyle{plain}
\Addresses

\end{document}